\documentclass[12pt]{article}
\usepackage{amsfonts}
\usepackage{amssymb}
\usepackage{mathrsfs}
\usepackage{srcltx}
\usepackage{amsmath}
\usepackage{amsthm}
\usepackage{amstext}
\usepackage{amsopn}
\usepackage{graphicx}

\newtheorem{theorem}{Theorem}[section]
\newtheorem{lemma}[theorem]{Lemma}

\theoremstyle{definition}

\theoremstyle{remark}

\numberwithin{equation}{section}

\newcommand{\ba}{\begin{array}}
\newcommand{\ea}{\end{array}}

\begin{document}
\date{}
\title{ \bf\large{A Note on Homoclinic Orbits for Second Order Hamiltonian Systems}}
\author{Bingyu Li\textsuperscript{1}\ \ Fengying Li\textsuperscript{2}\footnote{Corresponding Author, Email: lify0308@163.com}\ \ Donglun Wu\textsuperscript{3}\ \ Shiqing Zhang\textsuperscript{3}
 \\
{\small \textsuperscript{1} College of Information Science and Technology, Chengdu University of Technology,\hfill{\ }}\\
\ \ {\small Chengdu, Sichuan, 610059, P.R.China.\hfill{\ }}\\
{\small \textsuperscript{2} School of Economic and Mathematics, Southwestern University of Finance and Economics,\hfill{\ }}\\
\ \ {\small Chengdu, Sichuan, 611130, P.R.China.\hfill {\ }}\\
{\small \textsuperscript{3} Department of Mathematics, Sichuan University, \hfill{\ }}\\
\ \ {\small Chengdu, Sichuan, 610064, P.R. China.\hfill {\ }} }
\maketitle
\begin{abstract}
{In this paper, we study the existence for the homoclinic orbits for the second order Hamiltonian systems. Under suitable conditions on the potential $V$, we apply the direct method of variations and the Fourier analysis to prove the existence of homoclinc orbits.}

 \noindent{\emph{Keywords}}: Homoclinic solutions, Direct method of variations, Fourier analysis.

 \noindent{\emph{2000 Mathematical Subject Classification}}: 34C15, 34C25, 58F
\end{abstract}

\section {Introduction and Main Results}
\setcounter{section}{1} \setcounter{equation}{0}

Since the pioneer work of Rabinowitz [32] in 1978, many papers used variational methods to study the existence of periodic solutions for Hamiltonian systems. In recent 30 years, variational methods are also widely applied to the existence of homoclinic orbits for Hamiltonian systems(for examples,[1-31,34-48]etc.),specially,in 1990, P.H. Rabinowitz[36] used Mountain Pass Lemma and approximation arguments of periodic solutions to prove the following theorem about the homoclinic orbits for super-quadratic second order Hamiltonian systems:
\begin{theorem}\label{th1}
Suppose the second order Hamiltonian system:
\begin{equation}\label{e1}
\ddot{q}(t)+V_q(t,q)=0,
\end{equation}
where $q\in R^n$ and $V$ satisfies:
\begin{enumerate}
\item[($V_1$):] $V(t,q)=-\frac{1}{2}(L(t)q,q)+W(t,q)$, where $L$ is a continuous $T-$periodic matrix valued function and $W\in C^1(R\times R^n,R)$ is $T-$periodic in $t$;
\item[($V_2$):] $L(t)$ is positive definite symmetric for all $t\in [0,T]$;
\item[($V_3$):] there is a constant $\mu >2$ such that $0<\mu W(t,q)\leq (q,W_q(t,q))$ for all $q\in R^n\setminus \{0\}$;
\item[($V_4$):] $W(t,q)=o(|x|)$ as $q\rightarrow 0$ uniformly for $t\in [0,T]$.
\end{enumerate}
Then (\ref{e1}) possesses a nontrivial homoclinic solution $q(t)$ emanating from zero such that $q\in W^{1,2}(R, R^n)$.
\end{theorem}

Different from earlier papers, we use Fourier analysis and direct variational method to study the existence of homoclinic orbits for sub-quadratic second order Hamiltonian system, we have the following new theorem:

\begin{theorem}\label{th1.2}
Suppose $V(t,q)=-\frac{1}{2}|q|^2+a(t)|q|^{\alpha}$, $1<\alpha<2$ and $a(t)$ satisfies

\begin{enumerate}
\item[($a_1$):] $a(t)\in C^0(R,R^+);$
\item[($a_2$):] $a(t)\in L^1(R)\bigcap L^2(R).$
\end{enumerate}
Then (\ref{e1}) has at least one non-zero homoclinic orbit $q(t)$ with $q(\pm\infty)=0$ and $\dot{q}(\pm\infty)=0$ .
\end{theorem}

\section{Some Lemmas}

 Firstly,let us introduce some notations:\\
$H^1(R,R^n)=W^{1,2}(R,R^n)$ with the norm $\|q\|_{H^1}=(\int_{-\infty}^{+\infty}(|q|^2+|\dot{q}|^2)dt)^{1/2}$;\\
$\hat{f}(t)=\int_{-\infty}^{+\infty}f(x)e^{-2\pi ixt}dx$ is the Fourier transform of $f(x)$;\\
$\breve{g}(t)=\int_{-\infty}^{+\infty}g(x)e^{2\pi ixt}dx$ is the Fourier inverse transform of $g(x)$;\\
$I(q)=\frac{1}{2}\int_{-\infty}^{+\infty}(|q|^2+|\dot{q}|^2)dt-\int_{-\infty}^{+\infty}a(t)|q|^{\alpha} dt$ is the functional corresponding to the system (\ref{e1}) with $V(t,q)=-\frac{1}{2}|q|^2+a(t)|q|^{\alpha}$, $1<\alpha<2$, and $q\in H^1(R,R^n)$.
Now, we list some  Lemmas which are necessary for the proof of Theorem 1.2:

\begin{lemma}
Under the conditions of Theorem 1.2, for any given $q\in W^{1,2}(R,R^n)$,we
have that $-\infty<I(q)<+\infty$,that is,$I(q)$ is well defined.
\end{lemma}
\begin{proof}
Since
\begin{eqnarray*}
|q|&=&|\check{\hat{q}}|\leq\int_{-\infty}^{+\infty}|\hat{q}|dt\nonumber\\
&\leq& (\int_{-\infty}^{+\infty}\frac{1}{1+t^2})^{\frac{1}{2}}(\int_{-\infty}^{+\infty}(1+t^2)|\hat{q}(t)|dt)^{\frac{1}{2}}\nonumber\\
&=&\pi^{\frac{1}{2}}\|q\|_{H^1},
\end{eqnarray*}
Hence we have embedding relation $H^{1,2}(R,R^n)\subset C^0(R,R^n)$.
For any given $q\in H^{1,2}(R,R^n)$, we have
$$0\leq\int_{-\infty}^{+\infty}a(t)|q|^{\alpha} dt\leq\|q\|_{L^{\infty}}^{\alpha}\int_{-\infty}^{+\infty}a(t)dt<+\infty,$$
which implies
$$-\infty<I(q)=\frac{1}{2}\int_{-\infty}^{+\infty}(|q|^2+|\dot{q}|^2)dt-\int_{-\infty}^{+\infty}a(t)|q|^{\alpha}dt<+\infty.$$
So $I(q)$ is well-defined.
\end{proof}

\begin{lemma}
Suppose that the conditions of Theorem 1.2 hold, then we have
\begin{eqnarray*}
I^{'}(q)v=\int_{-\infty}^{+\infty}((q,v)+(\dot{q},\dot{v}))dt-\alpha\int_{-\infty}^{+\infty}a(t)|q|^{\alpha-2}(q,v)
dt
\end{eqnarray*}
for all $q$, $v\in H^{1}$, which implies that
\begin{eqnarray*}
I^{'}(q)q=\int_{-\infty}^{+\infty}(|q|^2+|\dot{q}|^2)dt-\alpha\int_{-\infty}^{+\infty}a(t)|q|^{\alpha}dt,
\end{eqnarray*}
and $I'(q)$ is continuous.
\begin{proof}
 Set
\begin{eqnarray*}
I_{1}(q)=\frac{1}{2}\int_{-\infty}^{+\infty}(|q|^2+|\dot{q}|^2)dt\ \
\ \mbox{and}\ \ \
I_{2}(q)=\int_{-\infty}^{+\infty}a(t)|q|^{\alpha}dt.
\end{eqnarray*}
Then we can see that $I(q)=I_{1}(q)-I_{2}(q)$. By the property of
norm, we conclude that $I_{1}\in C^{1}(H^{1},R)$ and
\begin{eqnarray*}
I_{1}^{'}(q)v=\int_{-\infty}^{+\infty}((q,v)+(\dot{q},\dot{v}))dt
\end{eqnarray*}

Now we prove $I_2$ is Frechet differentiable on $H^{1,2}$,the proof used some
ideas of Coti Zelati-Rabinowitz[26].
Let $W(t,q)=a(t)|q|^{\alpha}$
Since $|W_q(t,x)|=\alpha a(t)|x|^{\alpha-1},$
so for any $\epsilon >0$,there exists $\rho>0$ such that when $|x|\leq\rho$,we have
$$|W_q(t,x)|\leq \epsilon a(t).$$
It is well known [33] that for any finite R,
$$\int_{-R}^RW(t,q)dt\in C^1(W^{1,2}([-R,R],R^n);R).$$
So there is $\delta<\min(\frac{\rho}{4},1)$ such that
for $\phi\in W^{1,2}(R,R^n)$ and $||\phi||\leq\delta$,we have
$$|\int_{-R}^R(W(t,q+\phi)-W(t,q)-W_q(t,q)\phi)dt|\leq\frac{\epsilon}{4}||\phi||.$$
Since $q\in W^{1,2}(R,R^n)$,so $q(t)\rightarrow 0$ as $t\rightarrow\pm\infty$,
and we can choose R so large so that for $|t|\geq R$ ,we have
$$|q(t)|\leq\frac{\rho}{4}.$$
For $\phi\in W^{1,2}(R,R^n)$,by [36] we have
$$||\phi||_{L^{\infty}}\leq 2^{1/2}||\phi||\leq\rho/2.$$
By Mean Value Theorem,there exists $\xi=q+\theta\phi$ such that
$$W(t,q+\phi)-W(t,q)=W_q(t,\xi)\phi$$
Since for $|t|\geq R$,we have
$$|\xi|\leq\frac{\rho}{4}+\frac{\rho}{2}<\rho.$$
So for $|t|\geq R$,we have
$$|W(t,q+\phi)-W(t,q)|=|W_q(t,\xi)\phi|\leq\epsilon\phi a(t).$$
Hence
$$\int_{|t|>R}|W(t,q+\phi)-W(t,q)|dt\leq \int_{|t|>R}\epsilon\phi a(t)dt$$
$$\leq ||\phi||_{L^{\infty}}\epsilon\int_{-\infty}^{+\infty}a(t)dt$$
$$\leq 2\epsilon\int_{-\infty}^{+\infty}a(t)dt||\phi||$$
Similarly, we also have
$$\int_{|t|>R}|W_q(t,q)\phi|dt\leq 2\epsilon\int_{-\infty}^{+\infty}a(t)dt||\phi||.$$
Hence $I_2$ is Frechet differentiable.\\
Now we prove $I_2$ is continuous.
Suppose$q_m\rightarrow q$ in $W^{1,2}(R,R^n)$ .Since
$$\sup_{||\phi||=1}|\int_{-\infty}^{+\infty}(W_q(t,q_m)-W_q(t,q))\phi|
\leq(\int_{-\infty}^{+\infty}|W_q(t,q_m)-W_q(t,q)|^2dt)^{1/2}.$$
For any given $\epsilon>0$,we can choose $R>0$ large enough so that
$|t|\geq R$ implies $|q(t)|\leq\rho,|q_m(t)|\leq\rho$ for m large and
$$|W_q(t,q)|\leq\epsilon a(t),$$

$$|W_q(t,q_m)|\leq\epsilon a(t).$$
Hence
$$\int_{-\infty}^{+\infty}|W_q(t,q_m)-W_q(t,q)|^2dt$$
$$\leq\int_{-R}^{R}|W_q(t,q_m)-W_q(t,q)|^2dt+2\epsilon^2\int_{-\infty}^{+\infty}|a(t)|^2dt.$$
Hence $I'_2$ is continuous.

\end{proof}
\end{lemma}

\begin{lemma}\label{l1.2}([47])
Let $X$ be a reflexive Banach space, $M\subset X$ be a weakly closed subset, $f: M\rightarrow R$ be weakly lower semi-continuous. If $f$ is coercive, that is, $f(x)\rightarrow+\infty$ as $\|x\|\rightarrow+\infty$, then $f$ attains its infimum on $M$.
\end{lemma}

\begin{lemma}\label{l2.2}
$I(q)=\frac{1}{2}\int_{-\infty}^{+\infty}(|q|^2+|\dot{q}|^2)dt-\int_{-\infty}^{+\infty}a(t)|q|^{\alpha} dt$ is coercive for $q(t)\in H^1(R,R^n)$ and $I(q)$ is bounded from below.
\end{lemma}

\begin{proof}
\begin{eqnarray*}
I(q)&=&\frac{1}{2}\int_{-\infty}^{+\infty}(|q|^2+|\dot{q}|^2)dt-\int_{-\infty}^{+\infty}a(t)|q|^{\alpha} dt\\
&\geq& \frac{1}{2}\|q\|^2_{H^1}-\pi^{\frac{\alpha}{2}}\int_{-\infty}^{+\infty}a(t)dt\|q\|^{\alpha} _{H^1}.
\end{eqnarray*}
We notice that $1<\alpha<2$, so $I(q)\rightarrow+\infty \ \ \text{when}\ \ \|q\|_{H^1}\rightarrow+\infty$, hence $I(q)$ is coercive on $H^1$.

Let  $\phi(x)=\frac{1}{2}x^2-Cx^{\alpha}$, $C> 0$, $x\in [0,+\infty)$, then the derivative of $\phi$ is
$$\phi'(x)=x-C\alpha x^{\alpha-1}.$$
when $x=(C\alpha)^{\frac{1}{2-\alpha}}$, $\phi(x)$ attains its minimum, that is, when $\|q\|_{H^1}=(C\alpha)^{\frac{1}{2-\alpha}}$, $\frac{1}{2}\|q\|^2_{H^1}-\pi^{\frac{\alpha}{2}}\int_{-\infty}^{+\infty}a(t)dt\|q\|^{\alpha}_{H^1}$ attains its minimum, where $C= \pi^{\frac{\alpha}{2}}\int_{-\infty}^{+\infty}a(t)dt$. Then $I(q)$ is bounded from below.
\end{proof}

\begin{lemma}\label{12.3}(Sobolev-Rellich-Kondrachov [47]) Let $\Omega$ is bounded domain, then
$$W^{1,2}(\Omega,R)\subset C(\Omega,R)$$
and the embedding is compact.
\end{lemma}

\begin{lemma}\label{l2.4}([47])
Let $X$ be a Banach space, then the norm $\|\cdot\|$ is weakly lower semi-continuous.
\end{lemma}

\begin{lemma}\label{l2.5}
Assume $V(t,q)$ is defined as in Theorem 1.2, then $I(q)$ is weakly lower semi-continuous.
\end{lemma}

{\bf Proof}:\
We define \\
$I_k(q)=\frac{1}{2}\int_{-k}^{k}(|q|^2+|\dot{q}|^2)dt-\int_{-k}^{k}a(t)|q|^{\alpha}dt$,\\
$\tilde{q}_l(t)=q_l(t)\chi_{[-k,k]}$,\\
$\tilde{q}(t)=q(t)\chi_{[-k,k]}$.\\
We divide the proof into three steps:

Step 1: if $q_l(t)\rightharpoonup q(t)$ in $H^1(R,R^n)$, then $\tilde{q}_l(t)\rightharpoonup \tilde{q}(t)$ in $H^1([-k,k],R^n)=H^1[-k.k]$.

For any $\tilde{f}(t)\in H^1([-k,k],R^n)$, define
 \begin{equation*}
 f(x)=
\begin{cases}
\displaystyle \tilde{f}(t),  &\;\; t\in [-k,k];\\
\displaystyle 0, &\;\; x\notin [-k,k].
\end{cases}
\end{equation*}
Then
$$\int_{-k}^k\tilde{f}(t)\tilde{q}_l(t)dt=\int_{-\infty}^{+\infty} f(t)q_l(t)dt,$$

$$\int_{-k}^k\dot{\tilde{f}}(t)\dot{\tilde{q}}_l(t)dt=\int_{-\infty}^{+\infty}\dot{f}(t)\dot{q}_l(t)dt,$$

$$\int_{-k}^k\tilde{f}(t)\tilde{q}(t)dt=\int_{-\infty}^{+\infty} f(t)q(t)dt,$$

$$\int_{-k}^k\dot{\tilde{f}}(t)\dot{\tilde{q}}(t)dt=\int_{-\infty}^{+\infty}\dot{f}(t)\dot{q}(t)dt.$$
Since $q_l(t)\rightharpoonup q(t)$ in $H^1(R,R^n)$,then by Riesz representation Theorem,we have that
$$\int_{-\infty}^{+\infty}f(t)q_l(t)dt+\int_{-\infty}^{+\infty}\dot{f}(t)\dot{q}_l(t)dt\rightarrow
\int_{-\infty}^{+\infty}f(t)q(t)dt+\int_{-\infty}^{+\infty}\dot{f}(t)\dot{q}(t)dt,$$
then
\begin{equation*}
\int_{-k}^k\tilde{f}(t)\tilde{q}_l(t)dt+\int_{-k}^k\dot{\tilde{f}}(t)\dot{\tilde{q}}_l(t)dt\rightarrow
\int_{-k}^k\tilde{f}(t)\tilde{q}(t)dt+\int_{-k}^k\dot{\tilde{f}}(t)\dot{\tilde{q}}(t)dt,
\end{equation*}
that is, $\tilde{q}_l(t)\rightharpoonup \tilde{q}(t)$ in $H^1[-k,k]$.

Step 2: if $q_l(t)\rightharpoonup q(t)$ in $H^1(R,R^n)$, then by Step 1,we have that $\tilde{q}_l(t)\rightharpoonup \tilde{q}(t)$ in $H^1([-k,k],R^n)=H^1[-k.k]$.
By compact embedding theorem (Lemma 2.3) and functional analysis, $\tilde{q}_l(t)\rightarrow \tilde{q}(t)$
 uniformly in $C([-k,k],R^n)$,then we have
\begin{equation*}
\lim_{l\rightarrow+\infty}(\int_{-k}^k a(t)|\tilde{q}_l|^{\alpha}dt)\rightarrow \int_{-k}^{+k}a(t)|\tilde{q}(t)|^{\alpha}dt.
\end{equation*}
We have known that for $q,q_l\in W^{1,2}(R,R^n)$,when $t\rightarrow\pm\infty,$
$$q(t)\rightarrow 0,q_l(t)\rightarrow 0.$$
So by $a(t)\in L^1(R,R)$ we have
$$\int_{|t|\geq k}a(t)|q|^{\alpha}dt\leq\sup_{|t|\geq k}|q(t)|^{\alpha}\int_{|t|\geq k}a(t)dt\rightarrow 0,k\rightarrow +\infty.$$

$$\int_{|t|\geq k}a(t)|q_l|^{\alpha}dt\leq\sup_{|t|\geq k}|q_l(t)|^{\alpha}\int_{|t|\geq k}a(t)dt\rightarrow 0,k\rightarrow +\infty.$$
Hence
\begin{equation*}
\lim_{l\rightarrow+\infty}\lim_{k\rightarrow+\infty}(-\int_{-k}^k a(t)|\tilde{q}_l|^{\alpha}dt)
=\lim_{l\rightarrow+\infty}-\int_{-\infty}^{+\infty}a(t)|q_l|^{\alpha}dt
= -\int_{-\infty}^{+\infty}a(t)|q(t)|^{\alpha}dt.
\end{equation*}
Since the norm is weakly lower semi-continuous (w.l.s.c.), so we
have $\underline{\lim}\|q_l\|_{H^1(R,R^n)}\geq\|q\|_{H^1(R,R^n)}$.
So $\lim\inf _{l\rightarrow+\infty}I(q_l)\geq I(q)$. Then we apply
the Lemma 2.1 to get a minimizer for $I(q)$ on $H^{1,2}(R,R^n)$.

Step 3: Now we need to prove the minimizer is non-zero and
$q(\pm\infty)=0$ and $\dot{q}(\pm\infty)=0$. For any given $q_0\in
H^1(R,R^n)$,since $1<\alpha<2$, so for $r>0$ small enough,we have
$$I(rq_0)=\frac{1}{2}\int_{-\infty}^{+\infty}r^2(|q_0|^2+|\dot{q}_0|^2)dt-
r^{\alpha}\int_{-\infty}^{+\infty}a(t)|q_0|^{\alpha} dt<0.$$
So
$$min_{q\in H^1(R,R^n)}I(q)<0,$$
and the minimum value must be negative, hence the minimizer must be
non-zero. Similar to Rabinowitz [36] P.37, we can prove the
minimizer $q(t)$ satisfies that
$$q(\pm\infty)=0,\dot{q}(\pm\infty)=0.$$
By $q\in H^1(R,R^n)$,we have
$$\int_{|t|\geq A}(|q(t)|^2+|\dot{q}(t)|^2)dt\rightarrow 0.$$
By [36],we have
$$|q(t)|\leq 2[\int_{t-\frac{1}{2}}^{t+\frac{1}{2}}(|\dot{q}(s)|^2+|q(s)|^2)ds]^{1/2}.$$
Hence $q(\pm\infty)=0.$

By [36],we have
$$|\dot{q}(t)|\leq 2[\int_{t-\frac{1}{2}}^{t+\frac{1}{2}}(|\ddot{q}(s)|^2+|\dot{q}(s)|^2)ds]^{1/2}.$$
So if $\int_A^{A+1}|\ddot{q}(t)|^2dt\rightarrow 0$, as $A\rightarrow
+\infty$,we have $\dot{q}(\pm\infty)=0.$ Since $q(t)$ is the
solution of (1.1), so we have
$$|\ddot{q}(t)|^2=|-V_q(t,q)|^2\leq 2(|q(t)|^2+\alpha^2 |a(t)|^2|q(t)|^{2(\alpha-1)})$$
Since we have proved $q(t)\rightarrow 0$ as $ t\rightarrow \pm\infty,$
and by $a(t)\in L^2$,we can have
$$\int_A^{A+1}(|q(t)|^2+\alpha^2 |a(t)|^2|q(t)|^{2(\alpha-1)})dt\rightarrow 0.$$
$$\dot{q}(\pm\infty)=0.$$


\begin{thebibliography}{9}

\bibitem{1} A. Ambrosetti and M.Badiale,Homoclinic:Poincare-Melnikov type results via a variational approach,Ann.Inst.Henri,Poincare 15(1998),233-252.
\bibitem{2} A. Ambrosetti and M. L. Bertotti, Homoclinics for second order conservative systems, in "Partial Differential Equations and Related Subjects" (M. Miranda, Ed.), Research Notes in Mathematics, Longman, Harlow/New York, 1992, 21-37.
\bibitem{3} A. Ambrosetti and V. Coti Zelati, Multiple homoclinic orbits for a class of conservative systems,Rend. Sem. Math. Univ. Padova 89 (1993), 177-194.
\bibitem{4}V. I. Arnold, V. V. Kozlov, and A. I. Neishtadt, "Dynamical Systems III," VINITI, Moscow, 1985; English translation by Springer-Verlag, New York/Heidelberg/Berlin, 1988.
\bibitem{5}V. Benci and F. Giannoni, Homoclinic orbits on compact manifolds, J. Math. Anal. Appl. 157 (1991), 568-576.
\bibitem{6} M. L. Bertotti, Homoclinics for Lagrangian systems on Riemannian manifolds, Dynam. Systems Appl.1 (1992), 341-368.
\bibitem{7} M. L. Bertotti and S. V. Bolotin, A variational approach for homoclinics in almost periodic Hamiltonian system, Comm. Appl. Nonlinear Anal. 2 (1995), 43-57.
\bibitem{8} M. L. Bertotti and L. Jeanjean, Multiplicity of homoclinic solutions for singular second order conservative systems, Proc. Roy. Soc. Edinburgh Sect. A. 126(1996),1169-1180.
\bibitem{9} U. Bessi, Multiple homoclinic orbits for autonomous singular potentials, Proc. Roy. Soc. Edinburgh Sect. A 124 (1994), 785-802.
\bibitem{10} S. V. Bolotin, Libration motions of natural dynamical systems, Vestnik Moskov. Univ. Ser. I Mat. Mekh. 6 (1978), 72-77 (in Russian).
\bibitem{11} S. V. Bolotin, "Libration Motions of Reversible Hamiltonian Systems," dissertation, Moscow State University, Moscow, 1981 (in Russian).
\bibitem{12} S. V. Bolotin, Existence of homoclinic motions, Vestnik Moskov. Univ. Ser. I Mat. Mekh. 6 (1983), 98-103 (in Russian).
\bibitem{13} S. V. Bolotin, Homoclinic orbits to invariant tori in the perturbation theory for Hamiltonian systems,Prikl. Mat. Mekh. 54, No. 3 (1990), 497-502; English translation J. Appl. Math. Mech. 54, No. 3 (1990), 412-417.
\bibitem{14} S. V. Bolotin, Homoclinic orbits to minimal tori of Lagrangian systems, Vestnik. Moskov. Univ. Ser. I Mat. Mekh. 6 (1992), 34-41 (in Russian).
\bibitem{15} S. V. Bolotin, "Homoclinic Orbits to an Invariant Tori of Hamiltonian Systems," CARR Reports in Mathematical Physics, No. 27, L'Aquila, Italy, 1993.
\bibitem{16} S. V. Bolotin, Variational criteria for non integrability and chaos in Hamiltonian systems, in"Hamiltonian Systems: Integrability and Chaotic Behavior," Kluwer Academic, Dordrecht/Norwell, MA, 1994.
\bibitem{17} S. V. Bolotin and V. V. Kozlov, Libration in systems with many degrees of freedom, Prikl. Mat. Mekh.42, No. 2 (1978), 245¨C250; English translation in J. Appl. Math. Mech. 42, No. 2 (1978), 256-261.
\bibitem{18} S. V. Bolotin and P.Hybrid mountain pass homoclinic solution of a class of semilinear elliptical PDEs,
  Ann.Inst.H.Poincare Anal.Nonlineaire 31(2014),103-128.
\bibitem{19} B. Buffoni, Infinitely many large amplitude homoclinic orbits for a class of autonomous Hamiltonian systems, J. Differential Equations 121(1995),109-120.
\bibitem{20} B. Buffoni and E. Sere, A global condition for quasi random behavior in a class of conservative systems, Comm.Pure Appl.Math.49(1996),285-305.
\bibitem{21} P. Caldiroli and C. De Coster, Multiple homoclinics for a class of singular Hamiltonian systems, J. Math. Anal. Appl. 211 (1997), 556-573.
\bibitem{22} P. Caldiroli and L. Jeanjean, Homoclinics and heteroclinics for a class of conservative singular Hamiltonian systems, J. Differential Equations 136(1997),76-114.
\bibitem{23} P. Caldiroli and M. Nolasco, Multiple homoclinic solutions for a class of autonomous singular systems in R2, Ann. Inst. H. Poincar¨¦ Anal. Non Lin¨¦aire 15(1998),113-125.
\bibitem{24} K. Cieliebak and E. Sere, Pseudo-holomorphic curves and multiplicity of homoclinic orbits,
Duke Math.J.77(1995),483-518.

\bibitem{25} V. Coti Zelati, I. Ekeland, and E. Sere, A variational approach to homoclinic orbits in Hamiltonian systems, Math. Ann. 288 (1990), 133-160.

\bibitem{26} V. Coti Zelati and P. H. Rabinowitz, Homoclinic orbits for second order hamiltonian systems possessing superquadratic potentials, J. Amer. Math. Soc. 4 (1991), 693-727.

\bibitem{27}R. Giambo and F. Giannoni and P.Piccione,Multiple brake orbits and homoclinic orbits in Riemannian manifolds,Arch.Rational Mech.Anal.200(2011),691-724.
\bibitem{28} F. Giannoni and P. H. Rabinowitz, On the multiplicity of homoclinic orbits on Riemannian manifolds for a class of second order Hamiltonian systems, Nonlinear Differential Equations and Appl. 1(1993), 1-46.
\bibitem{29} H. Hofer and K. Wysocki, First order elliptic systems and the existence of homoclinic orbits in Hamiltonian systems, Math. Ann. 288 (1990), 483-503.
\bibitem{30} V. V. Kozlov, Calculus of variations in large and classical mechanics, Uspekhi Mat. Nauk 40 (1985), 33-60; English translation in Russian Math. Surveys 40 (1985),37-71.

\bibitem{31} P. Montecchiari, M. Nolasco, and S. Terracini, Multiplicity of homoclinics for a class of time recurrent second order Hamiltonian systems,Calc.Var. PDE 5(1997),523-555.
\bibitem{32} P. H. Rabinowitz, Periodic solutions of Hamiltonian systems,Comm.Pure and Appl.Math.31(1978),157-184.
\bibitem{33} P. H. Rabinowitz, Minimax methods in critical point theory and applications with applications to differential equations,CBMS,AMS 65,1986.
\bibitem{34} P. H. Rabinowitz,periodic and heteroclinic solutions for a periodic Hamiltonian systems,Ann.Henri Poincare Anal.Nonlineaire 6(1989),331-346.
\bibitem{35} P. H. Rabinowitz, Homoclinic and heteroclinic for a class of Hamiltonian systems,Calc. Var.1(1993),1-36.
\bibitem{36} P. H. Rabinowitz, Homocliniz orbits for a class of Hamiltonian systems, Proc. Roy. Soc. Edinburgh Sect. A 114 (1990), 33-38.
\bibitem{37} P. H. Rabinowitz, Homoclinic and heteroclinic orbits for a class of Hamiltonian systems, in "Calculus of Variations and PDE1" (1992).
\bibitem{38} P. H. Rabinowitz, Homoclinics for a singular Hamiltonian system, in "Geometric Analysis and the Calculus of Variations" (J. Jost, Ed.), International Press.
\bibitem{39} P. H. Rabinowitz, Homoclinics for an almost periodically forced singular Hamiltonian system, Topol. Methods Nonlinear Anal.6(1995),49-66.
\bibitem{40} P. H. Rabinowitz, Multibump solutions for an almost periodically forced singular Hamiltonian system,Electron J. Differential Equations 12 (1995).
\bibitem{41} P. H. Rabinowitz and K. Tanaka, Some results on connecting orbits for a class of Hamiltonian systems, Math. Z. 206 (1991), 473-499.
\bibitem{42} E. Sere, Existence of infinitely many homoclinics in Hamiltonian systems, Math. Z. 209 (1992), 27-42.
\bibitem{43} E. Sere, Looking for the Bernoulli shift, Ann. Inst. H. Poincar¨¦ Anal. Non Lin¨¦aire 10 (1993), 561-590.
\bibitem{44} K. Tanaka, Homoclinic orbits for a singular second order Hamiltonian system, Ann. Inst. H. Poincar¨¦ Anal. Non Lin¨¦aire 7(1990), 427-438.
\bibitem{45} K. Tanaka, A note on the existence of multiple homoclinic orbits for a perturbed radial potential, Nonlinear Differential Equations Appl. 1 (1994), 149-162.
\bibitem{46} K. Tanaka, Homoclinic orbits in a first order superquadratic Hamiltonian system: Convergence of subharmonic orbits, J. Differential Equations 94 (1991),315-339.
\bibitem{47} G.Q.Zhang(K.C.Chang),Lectures on variational methods,Higher Education Press,2011.
\bibitem{48} S.Q.Zhang, Symmetrically homoclinic orbits for symmetric Hamiltonian systems, JMAA 247(2000), 645-652.

\end{thebibliography}
\end{document}